\documentclass[12pt]{article}
\usepackage{amsthm,amsfonts,amsmath,amssymb,bm}
\usepackage[mathscr]{eucal}
\usepackage[cp1251]{inputenc}
\usepackage[english]{babel}
\usepackage{color}
\usepackage[final]{graphicx}

\newtheorem{theorem}{Theorem}
\newtheorem{proposition}{Proposition}

\begin{document}

\title{On the rate of convergence of $\mathbb{R}^d$-ergodic averages \\ constructed over a strictly convex set}       

\author{Ivan \uppercase{Podvigin}}             


\maketitle%

{{\bf Abstract:} For $\mathbb{R}^d$-ergodic means constructed over a strictly convex set, a spectral criterion for homogeneous rates of convergence is obtained. From Hertz's result on the asymptotics of the Fourier transform of the indicator of a strictly convex set, it follows that the rate of convergence does not depend on the specific form of such a set.}



\section{Introduction} 
The paper is devoted to the study of convergence rates in the statistical ergodic theorem.
It is well known that these rates are determined by a singularity of the spectral measure in the neighborhood of zero (see reviews~\cite{Ka96, KaPo16}, as well as recent works~\cite{Tem22, Bras24}). Recently, averaging with faster convergence rates than in classical ergodic theorems has also attracted interest (see,  e.g.,~\cite{DY18,DM23,TongLi24,P24Ufa}).

Let $\mathcal{H}$ be a Hilbert space on which the group~${\mathbb{R}^d}$ acts by unitary transformations~${U_{\bf t}},$ ${{\bf t}\in\mathbb{R}^d}$. Let ${\mathcal{K}\subset\mathbb{R}^d}$ be a set of finite positive Lebesgue measure (we denote it by~${\mathcal{L}_d}$), i.e., ${0<\mathcal{L}_d(\mathcal{K})<\infty}.$ The homothety generated by the vector~${{\bf t}\in\mathbb{R}^d}$ will be denoted as
$$
{\bf x}\mapsto {\bf x}\odot{\bf t}:=(x_1t_1,x_2t_2,\dots,x_dt_d),\ \ {\bf x}\in\mathbb{R}^d.
$$
For ${{\bf t}\in\mathbb{R}^d}$, we set $\mathcal{K}_{\bf t}=\mathcal{K}\odot {\bf t}.$ Clearly, ${\mathcal{L}_d(\mathcal{K}_{\bf t})=\mathcal{L}_d(\mathcal{K})t_1t_2\cdots t_d}.$ The statistical ergodic theorem states (see, e.g.,~\cite{Tem92}) that for any vector~${h\in\mathcal{H}}$
$$
\left\|\frac{1}{\mathcal{L}_d(\mathcal{K}_{\bf t})}\int_{\mathcal{K}_{\bf t}}U_{\bf s}h\,d\mathcal{L}_d({\bf s})-Ph\right\|_{\mathcal{H}}\to0
$$
for ${t_1,...,t_d\to\infty},$ where $P$ is the orthogonal projection onto the subspace of fixed vectors of the group $\{U_{\bf t}\}_{{\bf t}\in\mathbb{R}^d}.$

Of natural interest is the problem of the rate of convergence in this statement. In recent papers~\cite{Po23} and~\cite{KPTKh24, Po24}, power-law estimates of the convergence rate in this theorem have been obtained from the power-law singularity of the spectral measure at zero for two different cases: 1) the set ${\mathcal{K}=B({\bf0},1)}$ is a unit ball centered at zero, and the vector ${{\bf t}\in\mathbb{R}^d}$ goes to infinity, having all the same coordinates; and 2) the set ${\mathcal{K}=[0,1]^d}$ is a standard unit cube, and the vector~${{\bf t}\in\mathbb{R}^d}$ goes to infinity having coordinates that differ from each other. A significant difference for obtaining estimates of the convergence rate of ergodic averages in these two cases is the asymptotic behavior at infinity of the Fourier transform for the indicator of the set $\mathcal{K}$. The influence of this asymptotic behavior on the convergence rate in the statistical ergodic theorem can be seen thanks to the formula (see, e.g.,~\cite{BE74}):
\begin{equation}\label{eq_Norm}
\left\|\frac{1}{\mathcal{L}_d(\mathcal{K}_{\bf t})}\int_{\mathcal{K}_{\bf t}}U_{\bf s}h\,d\mathcal{L}_d({\bf s})-Ph\right\|^2_{\mathcal{H}}=\int_{\mathbb{R}^d}\left|\frac{1}{\mathcal{L}_d(\mathcal{K}_{\bf t})}\mathscr{F}[I_{\mathcal{K}_{\bf t}}]({\bf x})\right|^2d\sigma_{h-Ph}(\bf x),
\end{equation}
where ${\mathscr{F}[I_{\mathcal{K}_{\bf t}}]({\bf x})=\int_{{\mathcal{K}_{\bf t}}}e^{i({\bf y},{\bf x})}\,d\mathcal{L}_d(\bf y)}$ is the Fourier transform of~${I_{\mathcal{K}_{\bf t}}},$ ${({\bf x},{\bf y})=\sum\limits_{k=1}^dx_ky_k}$ is the standard inner product in $\mathbb{R}^d,$ and $\sigma_{h-Ph}$ is the spectral measure for the vector $h-Ph$ (see, e.g.,~\cite{F95}); in addition ${\sigma_{h-Ph}(\{{\bf 0}\})=0}.$ Considering that
\begin{multline*}
\frac{1}{\mathcal{L}_d(\mathcal{K}_{\bf t})}\mathscr{F}[I_{\mathcal{K}_{\bf t}}]({\bf x})=\frac{1}{\mathcal{L}_d(\mathcal{K}\odot{\bf t})}\int_{\mathcal{K}}e^{i({\bf x},{\bf y}\odot{\bf t})}d\mathcal{L}_d({\bf y}\odot{\bf t})=\\
=\frac{1}{\mathcal{L}_d(\mathcal{K})}\int_{\mathcal{K}}e^{i({\bf x}\odot{\bf t},{\bf y})}d\mathcal{L}_d({\bf y})=\frac{1}{\mathcal{L}_d(\mathcal{K})}\mathscr{F}[I_{\mathcal{K}}]({\bf x}\odot{\bf t}),
\end{multline*}
the integral we are interested in on the right side of formula~\eqref{eq_Norm} is written as
\begin{equation}\label{eq_StudInegral}
I_{\mathcal{K}}({\bf t})=\int_{\mathbb{R}^d}\left|\frac{1}{\mathcal{L}_d(\mathcal{K})}\mathscr{F}[I_{\mathcal{K}}]({\bf x}\odot{\bf t})\right|^2d\sigma_{h-Ph}(\bf x).
\end{equation}

In this paper we extend the results of~\cite{Po23}. Namely, we assume that the set $\mathcal{K}$ is convex, compact, and has a sufficiently smooth boundary~${\partial\mathcal{K}}$ with everywhere positive total (Gaussian) curvature, which we denote by~$\kappa.$ The asymptotics as ${|{\bf x}|:=\sqrt{({\bf x},{\bf x})}\to\infty}$ of the Fourier transform ${\mathscr{F}[I_{\mathcal{K}}]({\bf x})}$ of such a set is well studied (see, for example, the monographs~\cite[Chapter~VIII]{S93} and \cite[\S 2.3]{IL14} and the references therein). Namely, let the boundary $\partial\mathcal{K}$ belong to the smoothness class~$C^k$ for ${k>\max\left\{1,\frac{d-1}{2}\right\}}$, then there exists a constant ${D=D(\mathcal{K})>0}$ such that
\begin{equation}\label{eq_EstimationforFourier}
|\mathscr{F}[I_{\mathcal{K}}]({\bf x})|\leq D|{\bf x}|^{-\frac{d+1}{2}}.
\end{equation}
Moreover, the asymptotic equality proved by Hertz holds (see the more general result in~\cite[Theorem~2.29]{IL14}):
\begin{equation}\label{eq_AsimpforFourier}
\mathscr{F}[I_{\mathcal{K}}]({\bf x})=\mathcal{D}({\bf x})|{\bf x}|^{-\frac{d+1}{2}}+o(|{\bf x}|^{-\frac{d+1}{2}})
\end{equation}
as $|{\bf x}|\to\infty,$ where $\mathcal{D}({\bf x})$ is a bounded function of the following form.
We represent a nonzero vector ${{\bf x}\in\mathbb{R}^d}$ as ${{\bf x}=|{\bf x}|\bm\eta},$ where $\bm\eta$ is a unit vector specifying the direction. Denote by ${{\bf x}^{\pm}={\bf x}^{\pm}(\bm\eta)}$ the uniquely determined extremal points on the boundary~$\partial\mathcal{K}:$
$$
({\bf x}^+,\bm\eta)=\max_{{\bf y}\in\partial\mathcal{K}}({\bf y},\bm\eta),\ \ \ \
({\bf x}^-,\bm\eta)=\min_{{\bf y}\in\partial\mathcal{K}}({\bf y},\bm\eta);
$$
then
\begin{equation}\label{eq_Herz}
\mathcal{D}({\bf x})=\frac{1}{2\pi}\left(\frac{e^{2\pi i\left(|{\bf x}|({\bf x}^+,\bm\eta)+\frac{d+1}{8}\right)}}{\sqrt{\kappa({\bf x}^+)}}+\frac{e^{2\pi i\left(|{\bf x}|({\bf x}^-,\bm\eta)-\frac{d+1}{8}\right)}}{\sqrt{\kappa({\bf x}^-)}}\right).
\end{equation}

\section{Main results}

The power asymptotics~\eqref{eq_AsimpforFourier} of the Fourier transform leads to the maximum power with exponent $d+1$ decay rates to zero for the integral~${I_\mathcal{K}({\bf t})}$. In addition to the maximum power rates, we will find conditions on arbitrary power rates and even for a wider class of homogeneous functions.

Let $X$ be a cone set in ${\mathbb{R}^d},$ i.e., for every ${{\bf x}\in X}$ and every ${r>0}$ the point ${r{\bf x}\in X}.$ Recall that the mapping ${f:X\to\mathbb{R}}$ is called \textit{homogeneous of degree}~$\theta\in\mathbb{R}$ if for every $r>0$ and for every ${{\bf x}\in X}$ the equality ${f(r{\bf x})=r^\theta f({\bf x})}$ holds. Denote also the unit sphere in $X$ as
$$
S_X=\{{\bf x}\in X: |{\bf x}|=1\}.
$$
It is clear that for a homogeneous function it is enough to know its values on $S_X.$

The first main result of this paper is a criterion for homogeneous rates of convergence of the integral~${I_{\mathcal{K}}(\bf t)}$.
To formulate it, we consider neighborhoods of zero in $\mathbb{R}^d$ in the form of ellipsoids. For ${{\bm \delta}\in\mathbb{R}^d},$ ${{\bm \delta}>\bf0},$ put ${\bm \delta}^{-1}=(\delta_1^{-1},\dots,\delta_d^{-1})$ and
$$
\mathcal{E}({\bm \delta})=\{{\bf x}\in\mathbb{R}^d:\ \ |{\bf x}\odot{\bm \delta}^{-1}|<1\}=\left\{{\bf x}\in\mathbb{R}^d:\ \frac{x_1^2}{\delta_1^2}+\cdots+\frac{x_d^2}{\delta_d^2}<1\right\}.
$$

\begin{theorem}\label{Th_Criterion1} Let ${\varphi:\mathbb{R}^d_+\to\mathbb{R}_+}$ be a homogeneous function of degree ${\theta>-(d+1)}.$ Then, for ${t_1,...,t_d\to\infty},$ the equivalence holds
	$$
	I_{\mathcal{K}}({\bf t})=\mathcal{O}(\varphi({\bf t}))\ \ \Leftrightarrow\ \ \sigma_{h-Ph}(\mathcal{E}({\bf t}^{-1}))=\mathcal{O}(\varphi({\bf t})).
	$$
\end{theorem}

The second result is a criterion for continuous homogeneous rates of degree ${\theta=-(d+1)}$ in the case where the vector~${{\bf t}\in\mathbb{R}^d}$ remains in the sector, i.e., there exists a number~${B\geq1}$ such that ${{\bf t}\in X_B},$
where
$$
X_B=\{{\bf t}\in\mathbb{R}^d: \ t_i\leq Bt_j\ \ \text{for all indices}\ \ 1\leq i,j\leq d\}.
$$
To emphasize the constant~$B$, we will also say that ${\bf t}$ remains in the $B$-sector.

\begin{theorem}\label{Th_Criterion2} Let ${\varphi:\mathbb{R}^d_+\to\mathbb{R}_+}$ be a continuous homogeneous function of degree ${\theta=-(d+1)}.$ Then, for ${t_1,...,t_d\to\infty}$ remaining in the sector, the equivalence holds
	$$
	I_{\mathcal{K}}({\bf t})=\mathcal{O}(\varphi({\bf t}))\ \ \Leftrightarrow\ \
	\int_{\mathbb{R}^d}\frac{d\sigma_{h-Ph}({\bf x})}{|{\bf x}|^{d+1}}<\infty.
	$$
\end{theorem}

This result is based on the fact that all positive continuous homogeneous functions of the same degree, restricted to a sector, will be equivalent (Proposition~\ref{ProposSubHomogeneous}).

The last main result says that there are no other homogeneous rates.

\begin{theorem}\label{Th_Criterion3}
	Let ${\varphi:\mathbb{R}^d_+\to\mathbb{R}_+}$ be a homogeneous function of degree ${\theta<-(d+1)}.$ Then, for ${t_1,...,t_d\to\infty},$ the equivalence holds
	$$
	I_{\mathcal{K}}({\bf t})=\mathcal{O}(\varphi({\bf t}))\ \ \Leftrightarrow\ \
	\sigma_{h-Ph}=0.
	$$
\end{theorem}

Without loss of generality, in the following proofs we will assume that ${Ph=0}.$

\section{Homogeneous rates of degree $\theta>-(d+1)$}

In this section we prove Theorem~\ref{Th_Criterion1} and obtain estimates of the rate of convergence from estimates of the spectral measure (Proposition~\ref{Propos_Estimates}). For the one-way implication in Theorem~\ref{Th_Criterion1} we prove a more general statement.

\begin{proposition}\label{Pr1}
Suppose ${I_{\mathcal{K}}({\bf t})\leq\varphi(\bf t)}$ for all ${{\bf t}>\bf 0},$ where $\varphi$ is an arbitrary function tending to zero as ${t_1,...,t_d\to\infty}$. Then, for any ${\varepsilon\in(0,1)},$ there exists a vector ${{\bf a}>\bf0}$ such that for all ${{\bf t}>\bf 0}$ 
$$
{\sigma_h(\mathcal{E}({\bf t}^{-1}))\leq\varepsilon^{-2}\varphi({\bf t}\odot{\bf a})}.
$$ 
\end{proposition}

\begin{proof}
	
Indeed, $\frac{1}{\mathcal{L}_d(\mathcal{K})}|\mathscr{F}[I_{\mathcal{K}}]({\bf x})|$ is a continuous mapping of the variable ${{\bf x}\in\mathbb{R}^d}$ taking values from the interval~$[0,1]$ and equal to one for ${{\bf x}=\bf0}$. Therefore, for any ${\varepsilon\in(0,1)},$ there is a vector ${{\bf a}>\bf0}$ such that the smallest value of this function on the ellipsoid ${\mathcal{E}({\bf a})}$ is not less than ${\varepsilon}.$
Then, for all ${{\bf t}>\bf0},$ we have
$$
\varepsilon\leq\min\limits_{{\bf y}\in \mathcal{E}({\bf a})}\frac{1}{\mathcal{L}_d(\mathcal{K})}|\mathscr{F}[I_{\mathcal{K}}]({\bf y})|=
\min\limits_{{\bf x}\in \mathcal{E}({\bf t}^{-1})}\frac{1}{\mathcal{L}_d(\mathcal{K})}|\mathscr{F}[I_{\mathcal{K}}]({\bf x}\odot{\bf t}\odot{\bf a})|.
$$	
Taking this inequality into account and using~\eqref{eq_StudInegral}, we obtain
\begin{multline*}
\sigma_h(\mathcal{E}({\bf t}^{-1}))=\varepsilon^{-2}\int_{\mathcal{E}({\bf t}^{-1})}\varepsilon^2 d\sigma_h({\bf x})\leq\varepsilon^{-2}\int_{\mathcal{E}({\bf t}^{-1})}\frac{1}{\mathcal{L}^2_d(\mathcal{K})}|\mathscr{F}[I_{\mathcal{K}}]({\bf x}\odot{\bf t}\odot{\bf a})|^2d\sigma_h({\bf x})\leq\\
\leq\varepsilon^{-2}\int_{\mathbb{R}^d}\frac{1}{\mathcal{L}^2_d(\mathcal{K})}|\mathscr{F}[I_{\mathcal{K}}]({\bf x}\odot{\bf t}\odot{\bf a})|^2d\sigma_h({\bf x})=\varepsilon^{-2}I_{\mathcal{K}}({\bf t}\odot{\bf a})\leq\varepsilon^{-2}\varphi({\bf t}\odot{\bf a}).
\end{multline*}
\end{proof}

\begin{proof}[Proof of Theorem~\ref{Th_Criterion1}]
Assume ${I_{\mathcal{K}}({\bf t})\leq \varphi({\bf t})}$ for all ${\bf t}>{\bm0},$ where $\varphi$ is a homogeneous function of degree $\theta.$ Then Proposition~\ref{Pr1} yields the estimate
$\sigma_h(\mathcal{E}({\bf t}^{-1}))\leq\varepsilon^{-1}\varphi({\bf t}\odot{\bf a}).$ It is easy to see that the vector ${\bf a}$ can be chosen to have the same coordinates, i.e., ${{\bf a}=a(1,1,...,1)},\ a>0.$ Then from homogeneity we get
$$
\sigma_h(\mathcal{E}({\bf t}^{-1}))\leq\varepsilon^{-2}a^\theta\varphi({\bf t}).
$$

Let us now prove the implication in the other direction, using the formula for the integral~${I_{\mathcal{K}}(\bf t)}$ through the distribution function
\begin{equation}\label{eq_StudInegral2}
I_{\mathcal{K}}({\bf t})=2\int_0^1u\sigma_h\left(\left\{{\bf x}\in\mathbb{R}^d:\ \frac{1}{\mathcal{L}_d(\mathcal{K})}|\mathscr{F}[I_{\mathcal{K}}]({\bf x}\odot{\bf t})|>u\right\}\right)\,du.
\end{equation}
Here, the asymptotics of the Fourier transform plays a significant role. Using  inequality~\eqref{eq_EstimationforFourier} and substituting it into~\eqref{eq_StudInegral2}, we obtain the estimate
\begin{multline*}
I_{\mathcal{K}}({\bf t})\leq2\int_0^1u\sigma_h\left(\left\{{\bf x}\in\mathbb{R}^d:\ \frac{1}{\mathcal{L}_d(\mathcal{K})}D|{\bf x}\odot{\bf t}|^{-\frac{d+1}{2}}>u\right\}\right)\,du=\\
=2\int_0^1u\sigma_h\left(\left\{{\bf x}\in\mathbb{R}^d:\ |{\bf x}\odot{\bf t}|<\left(\frac{D}{u\mathcal{L}_d(\mathcal{K})}\right)^{\frac{2}{d+1}}\right\}\right)\,du=\\
=2\int_0^1u\sigma_h\left(\mathcal{E}\left({\bf t}^{-1}\left(\frac{u\mathcal{L}_d(\mathcal{K})}{D}\right)^{-\frac{2}{d+1}}\right)\right)\,du.
\end{multline*}
Estimating the spectral measure on ellipsoids, ${\sigma_h(\mathcal{E}({\bf t}^{-1}))\leq\varphi({\bf t})}$ for any vector ${{\bf t}>\bf0},$ leads to the inequality
$$
I_{\mathcal{K}}({\bf t})\leq2\int_0^1u\varphi\left({\bf t}\left(\frac{u\mathcal{L}_d(\mathcal{K})}{D}\right)^{\frac{2}{d+1}}\right)\,du\leq
2\int_0^1u\varphi({\bf t})\left(\left(\frac{u\mathcal{L}_d(\mathcal{K})}{D}\right)^{\frac{2}{d+1}}\right)^\theta\,du.
$$
Denoting
$$
C=2\left(\frac{\mathcal{L}_d(\mathcal{K})}{D}\right)^{\frac{2\theta}{d+1}},
$$
for all ${{\bf t}>\bf0},$ we obtain
$$
I_{\mathcal{K}}({\bf t})\leq
C\varphi({\bf t})\int_0^1u^{1+\frac{2\theta}{d+1}}\,du=\varphi({\bf t})\frac{C(d+1)}{2(d+1+\theta)}.
$$
This inequality completes the proof of Theorem~\ref{Th_Criterion1}.
\end{proof}

As a simple consequence of the last calculations, the following statement can easily be obtained (see also~\cite[\S 3.1]{KPTKh24}).

\begin{proposition}\label{Propos_Estimates}
	Let ${\varphi:\mathbb{R}^d_+\to\mathbb{R}_+}$ be a homogeneous function of degree ${\theta\leq0},$ and 
	$$
	{\sigma_{h}(\mathcal{E}({\bf t}^{-1}))=\mathcal{O}(\varphi({\bf t}))}\ \ \ \text{as} \ \ \ {t_1,...,t_d\to\infty}.
	$$ 
	(1) If ${\theta\in(-(d+1),0]},$ then $I_{\mathcal{K}}({\bf t})=\mathcal{O}(\varphi({\bf t}));$
	\\
	(2) if ${\theta=-(d+1)},$ then $I_{\mathcal{K}}({\bf t})=\mathcal{O}(\varphi({\bf t})\ln\frac{1}{\varphi({\bf t})});$
	\\
	(3) if ${\theta<-(d+1)},$ then $I_{\mathcal{K}}({\bf t})=\mathcal{O}(\varphi^{-\frac{d+1}{\theta}}({\bf t})).$
\end{proposition}

\begin{proof} Case (1) is already proved.
Using~\eqref{eq_StudInegral2}, for ${\varepsilon\in(0,1)},$ we get
$$
I_{\mathcal{K}}({\bf t})\leq\varepsilon^2\|h\|^2+C\varphi({\bf t})\int_\varepsilon^1u^{1+\frac{2\theta}{d+1}}\,du.
$$
If ${\theta=-(d+1)},$ then we obtain the estimate
$$
I_{\mathcal{K}}({\bf t})\leq\varepsilon^2\|h\|^2-C\varphi({\bf t})\ln\varepsilon.
$$
Minimization over ${\varepsilon>0}$ yields, for all ${\bf t}>0$ with ${\varepsilon^2=\frac{C\varphi(\bf{t})}{2\|h\|^2}<1},$
$$
I_{\mathcal{K}}({\bf t})\leq\frac{C}{2}\varphi({\bf t})\left(1-\ln\left(\frac{C\varphi(\bf{t})}{2\|h\|^2}\right)\right)=\mathcal{O}\left(\varphi({\bf t})\ln\frac{1}{\varphi({\bf t})}\right).
$$
If now ${\theta<-(d+1)},$ then we obtain the estimate
$$
I_{\mathcal{K}}({\bf t})\leq\varepsilon^2\|h\|^2+\varphi({\bf t})\frac{C(d+1)}{2|d+1+\theta|}\varepsilon^{2(1+\frac{\theta}{d+1})}.
$$
Taking ${\varepsilon^{\frac{2\theta}{d+1}}=\frac{\|h\|^2}{C\varphi({\bf t})}}>1,$ we get
${I_{\mathcal{K}}({\bf t})=\mathcal{O}(\varphi^{-\frac{d+1}{\theta}}({\bf t}))}.$
\end{proof}

\section{Homogeneous rates of degree $\theta=-(d+1)$}

In the following statement we consider conditions under which 
all positive continuous homogeneous functions will be equivalent.

\begin{proposition}\label{ProposSubHomogeneous}
Let $X\subset\mathbb{R}^d$ be a cone set. If $S_X$ is a compact set, then all homogeneous functions of the same degree defined on $X$, continuous and positive on the unit sphere $S_X$, are equivalent. Conversely, if all homogeneous functions of the same degree defined on $X$, continuous and positive on the unit sphere $S_X$, are equivalent, then $S_X$ is a compact set.
\end{proposition}

\begin{proof}
Let $\varphi$ be a homogeneous function of degree~$\theta,$ continuous and positive on the unit sphere $S_X$. Show that there exist constants~${E,F>0}$ such that for all ${{\bf t}\in X}$
$$
E|{\bf t}|^{\theta}\leq\varphi({\bf t})\leq F|{\bf t}|^{\theta}.
$$
Indeed, since
$$
\varphi({\bf t})=|{\bf t}|^\theta\varphi\left(\frac{t_1}{|t|},\frac{t_2}{|t|},...,\frac{t_d}{|t|}\right),
$$
then
$$
E=\inf\limits_{{\bf s}\in S_X}\varphi({\bf s}),\ \ F=\sup\limits_{{\bf s}\in S_X}\varphi({\bf s}).
$$
Since $\varphi$ is continuous and positive on the unit sphere~$S_X$ and $S_X$ is a compact set, then ${0<E\leq F<\infty}.$

Let us now prove the converse statement. Let ${\bf s}_n$ be a sequence in $S_X.$ There is a subsequence ${\bf s}_{n_k}$ and a vector ${\bf s}\in\mathbb{R}^d$ such that ${|{\bf s}_{n_k}-{\bf s}|\to0}$ as ${k\to\infty}.$ Suppose that ${{\bf s}\not\in S_X},$ and consequently ${r{\bf s}\not\in X}$ for all $r>0.$ 
Without loss of generality, we will assume that ${s_1\neq0}.$
Consider the homogeneous function
$$
\varphi_{\bf s}({\bf t})=\left(\sum_{j=1}^d\bigl||t_1|s_j-t_j|s_1|\bigr|\right)^\theta,\ \ {\bf t}\in X
$$
of nonzero degree $\theta.$ It is well-defined on $X,$ positive and continuous since 
$$
\sum_{j=1}^d\bigl||t_1|s_j-t_j|s_1|\bigr|=0\ \ \ \Leftrightarrow\ \ \ {\bf t}=r{\bf s}
$$
for some $r>0.$
From the equivalence condition we have 
$$
E|{\bf t}|^{\theta}\leq\varphi_{\bf s}({\bf t})\leq F|{\bf t}|^{\theta}
$$
for all ${\bf t}\in X$ and some $E,F>0.$ For ${\bf t}={\bf s}_{n_k}$ we get
$$
E\leq\varphi_{\bf s}({\bf s}_{n_k})\leq F.
$$
Passing to the limit as $k\to\infty$, we get a contradiction, since
$$
\lim_{k\to\infty}\varphi_{\bf s}({\bf s}_{n_k})=
\left\{\begin{array}{cl}
0& \text{if}\ \theta>0\\
+\infty& \text{if}\ \theta<0.
\end{array}\right.
$$
For the case ${\theta=0}$ to obtain a contradiction, it is sufficient to take the function
$$
\varphi_{\bf s}(\bf t)=\left|{\bf s}-\frac{\bf t}{|\bf t|}\right|.
$$
\end{proof}

Note that if $X$ is a sector, then the unit sphere $S_X$ is a compact set.

\begin{proof}[Proof of Theorem~\ref{Th_Criterion2}]

Assume that ${A=\int_{\mathbb{R}^d}\frac{d\sigma_{h}({\bf x})}{|{\bf x}|^{d+1}}<\infty}.$
Taking into account equality~\eqref{eq_StudInegral} and  estimate~\eqref{eq_EstimationforFourier}, we have
$$
I_{\mathcal{K}}({\bf t})\leq\frac{D^2}{\mathcal{L}_d^2(\mathcal{K})}\int_{\mathbb{R}^d}\frac{d\sigma_h}{|{\bf x}\odot{\bf t}|^{d+1}}.
$$
The condition of belonging to the $B$-sector for the vector~${{\bf t}>\bf0}$ implies
$$
|{\bf x}|Bt_\ell\geq|{\bf x}\odot{\bf t}|\geq|{\bf x}|B^{-1}t_\ell,\ \ 1\leq\ell\leq d,
$$
and therefore
$$
\frac{B|{\bf t}||{\bf x}|}{\sqrt{d}}\geq|{\bf x}\odot{\bf t}|\geq\frac{B^{-1}|{\bf t}||{\bf x}|}{\sqrt{d}}.
$$
As a result, taking into account Proposition~\ref{ProposSubHomogeneous}, we obtain the required estimate
$$
I_{\mathcal{K}}({\bf t})\leq\frac{D^2B^{d+1}d^{\frac{d+1}{2}}}{\mathcal{L}_d^2(\mathcal{K})}\int_{\mathbb{R}^d}\frac{d\sigma_h}{|{\bf t}|^{d+1}|{\bf x}|^{d+1}}=\frac{D^2B^{d+1}d^{\frac{d+1}{2}}A}{\mathcal{L}_d^2(\mathcal{K})}|{\bf t}|^{-(d+1)}\leq\frac{D^2B^{d+1}d^{\frac{d+1}{2}}AE^{-1}}{\mathcal{L}_d^2(\mathcal{K})}\varphi({\bf t}).
$$

Assume now that ${I_{\mathcal{K}}({\bf t})=\mathcal{O}(\varphi({\bf t}))}$ for all ${{\bf t}>\bf0}$ from the sector~$X_B.$
It follows from Proposition~\ref{ProposSubHomogeneous} that
$$
I_{\mathcal{K}}({\bf t})\leq F|{\bf t}|^{-(d+1)}
$$
for some constant ${F>0}.$ Next we will use the proof scheme proposed by Robinson in~\cite{Rob}. Integrating over the part of the ball~${B({\bf0},p)}$ with center at zero and radius $p>0,$ lying in $X_B,$ we obtain
$$
\frac{1}{\mathcal{L}_d(B({\bf0},p))}\int_{B({\bf0},p)\cap X_B}|{\bm\tau}|^{d+1}I_{\mathcal{K}}({\bm\tau})\,d\mathcal{L}_d({\bm\tau})\leq F.
$$
Expression~\eqref{eq_StudInegral} for the integral ${I_{\mathcal{K}}}$ and Tonelli's theorem lead to the inequality
$$
\int_{\mathbb{R}^d}\frac{1}{p^d}\int_{B({\bf0},p)\cap X_B}|{\bm\tau}|^{d+1}\left|\mathscr{F}[I_{\mathcal{K}}]({\bf x}\odot{\bm\tau})\right|^2\,d\mathcal{L}_d({\bm\tau})\,d\sigma_h({\bf x})\leq F\mathcal{L}^2_d(\mathcal{K})\frac{\pi^{d/2}}{\Gamma(d/2+1)}.
$$
Applying Fatou's lemma to this inequality we obtain
\begin{equation}\label{eq_Int&Lim}
\int_{\mathbb{R}^d}\lim_{p\to\infty}\frac{1}{p^d}\int_{B({\bf0},p)\cap X_B}|{\bm\tau}|^{d+1}\left|\mathscr{F}[I_{\mathcal{K}}]({\bf x}\odot{\bm\tau})\right|^2\,d\mathcal{L}_d({\bm\tau})\,d\sigma_h({\bf x})\leq F\mathcal{L}^2_d(\mathcal{K})\frac{\pi^{d/2}}{\Gamma(d/2+1)}.
\end{equation}
We will formulate the calculation of the limit as an independent statement.

\begin{proposition}\label{Pr_limit}
For any ${\bf x}\neq{\bf 0}$ the equality holds
\begin{multline*}
\lim_{p\to\infty}\frac{1}{p^d}\int_{B({\bf0},p)^d\cap X_B}|{\bm\tau}|^{d+1}\left|\mathscr{F}[I_{\mathcal{K}}]({\bf x}\odot{\bm\tau})\right|^2\,d\mathcal{L}_d({\bm\tau})=\\
=\frac{1}{4\pi^2}\int_{B({\bf0},1)\cap X_B}\frac{|{\bf s}|^{d+1}}{|{\bf x}\odot{\bf s}|^{d+1}}
\left(\frac{1}{\kappa(({\bf x}\odot{\bf s})^+)}+\frac{1}{\kappa(({\bf x}\odot{\bf s})^-)}\right)\,d\mathcal{L}_d({\bf s}).
\end{multline*}	
\end{proposition}
\begin{proof}[Proof of Proposition~\ref{Pr_limit}]
Let's rewrite the integral by replacing the variables ${{\bm\tau}=p{\bf s}}:$
$$
\frac{1}{p^d}\int_{B({\bf0},p)\cap X_B}|{\bm\tau}|^{d+1}\left|\mathscr{F}[I_{\mathcal{K}}]({\bf x}\odot{\bm\tau})\right|^2\,d\mathcal{L}_d({\bm\tau})=\int_{B({\bf0},1)\cap X_B}p^{d+1}
|{\bf s}|^{d+1}\left|\mathscr{F}[I_{\mathcal{K}}]({\bf x}\odot p{\bf s})\right|^2\,d\mathcal{L}_d({\bf s}).
$$
Using~\eqref{eq_AsimpforFourier} we represent the integrand as the sum of two functions
\begin{equation}\label{eq_Asymp1}
p^{d+1}|{\bf s}|^{d+1}\left|\mathscr{F}[I_{\mathcal{K}}]({\bf x}\odot p{\bf s})\right|^2=
|{\bf s}|^{d+1}|\mathcal{D}(p({\bf x}\odot{\bf s}))|^2|{\bf x}\odot {\bf s}|^{-(d+1)}+T(p,{\bf x},{\bf s}),
\end{equation}
where
\begin{equation}\label{eq_Asymp2}
T(p,{\bf x}, {\bf s})=|{\bf s}|^{d+1}p^{d+1}\left(|t(p({\bf x}\odot{\bf s}))|^2+2|p({\bf x}\odot{\bf s})|^{-(d+1)/2}\mathrm{Re}\,(\mathcal{D}(p({\bf x}\odot{\bf s}))t(p({\bf x}\odot{\bf s}))\right),
\end{equation}
and ${t({\bf y})=o(|{\bf y}|^{-(d+1)/2}))}$ as ${|{\bf y}|\to\infty}.$
It is easy to see that ${{T(p,{\bf x}, {\bf s})}\to0}$ for ${p\to+\infty}.$ Moreover, it is  bounded by integrable function of the variable ${\bf s}.$ Therefore, applying Lebesgue's theorem on dominated convergence to the integral of $T,$ we obtain zero.

Taking into account equality~\eqref{eq_Herz} for the function $\mathcal{D}$ we obtain
\begin{equation}\label{eq_Asymp3}
|\mathcal{D}(p({\bf x}\odot{\bf s}))|^2=\frac{1}{4\pi^2}\left(\frac{1}{\kappa(({\bf x}\odot{\bf s})^+)}+\frac{1}{\kappa(({\bf x}\odot{\bf s})^-)}\right)+\frac{\cos2\pi h(p,{\bf x},{\bf s})}{2\pi^2\sqrt{\kappa(({\bf x}\odot{\bf s})^+)\kappa(({\bf x}\odot{\bf s})^-)}},
\end{equation} 
where ${h(p,{\bf x},{\bf s})=p\bigl(({\bf x}\odot{\bf s})^+-({\bf x}\odot{\bf s})^-,{\bf x}\odot{\bf s}\bigr)+\frac{d+1}{4}}.$
It remains to be noted that
$$
\lim_{p\to+\infty}\int_{B({\bf0},1)\cap X_B}\frac{|{\bf s}|^{d+1}}{|{\bf x}\odot{\bf s}|^{d+1}}\frac{\cos2\pi h(p,{\bf x},{\bf s})}{\sqrt{\kappa(({\bf x}\odot{\bf s})^+)\kappa(({\bf x}\odot{\bf s})^-)}}\,d\mathcal{L}_d({\bf s})=0.
$$
This is immediately visible if we switch to spherical coordinates: ${{\bf s}=(\rho,\eta)},$ where ${\rho=|{\bf s}|\in[0,1]},$ and ${\eta=\frac{{\bf s}}{|{\bf s}|}}.$ The inner integral over $\rho$ tends to zero as ${p\to\infty}$ according to the Riemann-Lebesgue lemma. And to the external integral over $\eta$ we again apply the Lebesgue theorem on dominated convergence.

Thus, Proposition~\ref{Pr_limit} is completely proved.
\end{proof}

Let us return to the inequality~\eqref{eq_Int&Lim}, having first estimated from below the resulting integral from Proposition~\ref{Pr_limit}. Considering that the curvature is bounded and the vector ${\bf s}$ is in the $B$-sector, for ${\bf x}\neq{\bf0}$ we obtain
\begin{multline*}
\frac{1}{4\pi^2}\int_{B({\bf0},1)\cap X_B}\frac{|{\bf s}|^{d+1}}{|{\bf x}\odot{\bf s}|^{d+1}}
\left(\frac{1}{\kappa(({\bf x}\odot{\bf s})^+)}+\frac{1}{\kappa(({\bf x}\odot{\bf s})^-)}\right)\,d\mathcal{L}_d({\bf s})\geq\\
\geq\frac{1}{2\pi^2\max\limits_{{\bf y}\in\partial\mathcal{K}}\kappa({\bf y})}\int_{B({\bf0},1)\cap X_B}\frac{|{\bf s}|^{d+1}d^{\frac{d+1}{2}}}{(B|{\bf s}||{\bf x}|)^{d+1}}\,d\mathcal{L}_d({\bf s})=\frac{d^{\frac{d+1}{2}}\mathcal{L}_d(B({\bf0},1)\cap X_B)}{2\pi^2\max\limits_{{\bf y}\in\partial\mathcal{K}}\kappa({\bf y})B^{d+1}|\bf x|^{d+1}}
\end{multline*}
Given this estimate and equality ${\sigma_h(\{{\bf 0}\})=0}$, the required statement follows from \eqref{eq_Int&Lim}:
$$
\int_{\mathbb{R}^d}\frac{d\sigma_h({\bf x})}{|{\bf x}|^{d+1}}\leq F\mathcal{L}^2_d(\mathcal{K})\frac{\pi^{d/2}}{\Gamma(d/2+1)}
\frac{2\pi^2\max\limits_{{\bf y}\in\partial\mathcal{K}}\kappa({\bf y})B^{d+1}}{d^{\frac{d+1}{2}}\mathcal{L}_d(B({\bf0},1)\cap X_B)}.
$$
Theorem~\ref{Th_Criterion2} is proved.
\end{proof}

Note that when the sector $X_B$ expands, i.e., when $B\to+\infty,$ the constant on the right-hand side of the resulting inequality goes to infinity. It is still unclear whether the condition of belonging to the sector when ${{t_1,...,t_d}\to+\infty}$ is essential or technical.

\section{Proof of Theorem~\ref{Th_Criterion3}}

Assume that ${I_{\mathcal{K}}({\bf t})=\mathcal{O}(\varphi({\bf t}))}$ as ${t_1,...,t_d\to\infty},$ where $\varphi$ is a positive homogeneous function of degree ${\theta<-(d+1)}.$ Taking ${\bf t}=p{\bf s}$ for some fixed unit vector ${\bf s}$ with positive coordinates as ${p\to+\infty}$ we obtain
$$
|{\bf t}|^{d+1}I_{\mathcal{K}}({\bf t})=|{{\bf t}}|^{\theta+d+1}\mathcal{O}\left(\varphi\left(\frac{t_1}{|t|},...,\frac{t_d}{|t|}\right)\right)=p^{\theta+d+1}\mathcal{O}\left(\varphi\left(s_1,...,s_d\right)\right)=o(1).
$$
It yields 
$$
\lim_{p\to+\infty}\int_{\mathbb{R}^d\setminus\{{\bf0}\}}p^{d+1}\left|\mathscr{F}[I_{\mathcal{K}}]({\bf x}\odot p{\bf s})\right|^2\,d\sigma_h({\bf x})=0.
$$
On the other hand, we can calculate this limit by taking into account~\eqref{eq_Asymp1}, \eqref{eq_Asymp2} and \eqref{eq_Asymp3}, where now ${\bf s}$ is a fixed unit vector, and the variable is ${\bf x}.$
We have
$$
\lim_{p\to+\infty}\int_{\mathbb{R}^d\setminus\{{\bf0}\}}|\mathcal{D}(p({\bf x}\odot{\bf s}))|^2|{\bf x}\odot {\bf s}|^{-(d+1)}\,d\sigma_h({\bf x})+\lim_{p\to+\infty}\int_{\mathbb{R}^d\setminus\{{\bf0}\}}T(p,{\bf x},{\bf s})\,d\sigma_h({\bf x})=0.
$$
The second limit is zero by Lebesgue's theorem on dominated convergence, since, as noted above, ${T(p,{\bf x},{\bf s})\to0}$ for ${p\to+\infty}$, and there is an integrable majorant. Namely, we know from Theorem~\ref{Th_Criterion2} that ${\int_{\mathbb{R}^d}\frac{d\sigma_h}{|{\bf x}|^{d+1}}<\infty},$ and
$$
|T(p,{\bf x},{\bf s})|\leq\frac{C}{|{\bf x}\odot{\bf s}|^{d+1}}\leq\frac{C}{(\min_k s_k|{\bf x}|)^{d+1}}.
$$
For the first limit using~\eqref{eq_Asymp3} we obtain
$$
\lim_{p\to\infty}\int\limits_{\mathbb{R}^d\setminus\{{\bf0}\}}
\left(\left(\frac{1}{\sqrt{\kappa(({\bf x}\odot{\bf s})^+)}}-\frac{1}{\sqrt{\kappa(({\bf x}\odot{\bf s})^-)}}\right)^2+\frac{4\cos^2\pi h(p,{\bf x},{\bf s})}{\sqrt{\kappa(({\bf x}\odot{\bf s})^+)\kappa(({\bf x}\odot{\bf s})^-)}}\right)\,\frac{d\sigma_h({\bf{x}})}{|{\bf x}\odot{\bf s}|^{d+1}}=0.
$$
It follows that
$$
\lim_{p\to+\infty}\int_{\mathbb{R}^d\setminus\{{\bf0}\}}\frac{\cos^2\pi h(p,{\bf x},{\bf s})}{\sqrt{\kappa(({\bf x}\odot{\bf s})^+)\kappa(({\bf x}\odot{\bf s})^-)}}\,\frac{d\sigma_h({\bf{x}})}{|{\bf x}\odot{\bf s}|^{d+1}}=0.
$$
Introducing on ${\mathbb{R}^d\setminus\{{\bf0}\}}$ a finite Borel measure 
$$
d\nu({\bf x})=\frac{d\sigma_h({\bf{x}})}{\sqrt{\kappa(({\bf x}\odot{\bf s})^+)\kappa(({\bf x}\odot{\bf s})^-)}|{\bf x}\odot{\bf s}|^{d+1}},
$$ 
we obtain that ${\cos(\pi pa({\bf x})+\pi(d+1)/4)}$ tends to zero as ${p\to+\infty}$ in the norm of the space ${L_2(\nu)},$  
where
$$
a({\bf x})=(({\bf x}\odot{\bf s})^+-({\bf x}\odot{\bf s})^-,{\bf x}\odot{\bf s})=|{\bf x}\odot{\bf s}|(\max_{{\bf y}\in\partial\mathcal{K}}({\bf y},\bm\eta)-\min_{{\bf y}\in\partial\mathcal{K}}({\bf y},\bm\eta))>0.
$$
Using similar arguments as in Lemma~4 of~\cite{KPT23}, we obtain that the measure $\nu$ and, consequently, $\sigma_h$ is zero. Theorem~\ref{Th_Criterion3} is proved.

\section{Comparison of estimates for classical ergodic means}

As we see, for ergodic averages constructed over a strictly convex set (with a sufficiently smooth boundary), the convergence rate estimates do not depend on the specific type of set. Therefore, we consider averages constructed over the unit ball and compare the results from~\cite{Po24} for averages constructed over the cube. For such averages, neighborhoods of zero in the form of parallelepipeds were considered
$$
\Pi({\bf t}^{-1})=\{x\in\mathbb{R}^d: -1<t_kx_k\leq1\}, \ {\bf t}>\bf0.
$$
The asymptotics of the spectral measure on such sets is equivalent to the asymptotics on the ellipsoids $\mathcal{E}({\bf t}^{-1}).$ Let us consider the scale of a power singularity, i.e.,
\begin{equation}\label{eq_PowerEstimate}
\sigma_h(\mathcal{E}({\bf t}^{-1}))\asymp\sigma_h(\Pi({\bf t}^{-1}))=\mathcal{O}({\bf t}^{-\bm \alpha}),
\end{equation}
where ${\bm\alpha=(\alpha_1,...,\alpha_d)>\bf 0}.$ It is clear that the function $\varphi({\bf t})={\bf t}^{-\bm \alpha}=t_1^{-\alpha_1}\cdots t_d^{-\alpha_d}$ is homogeneous of degree ${\theta=-(\alpha_1+...+\alpha_d)}.$

From Proposition~\ref{Propos_Estimates} and Theorem~2 from~\cite{Po24} we obtain the following estimates for the integral
$I_\square({\bf t})$ constructed over a cube and for $I_\bigcirc({\bf t})$ constructed over a ball.
For the parametric vector $\bm{\alpha}$ we denote by $\bm{\alpha}^*=(\alpha^*_1,...,\alpha^*_d)$ the permutation in increasing coordinates~$\bm{\alpha},$ i.e., ${\alpha^*_1\leq\alpha^*_2\leq ... \leq\alpha^*_d}.$
Among the successive differences ${a^*_k-a^*_{k+1}}$ for ${1\leq k \leq d-1}$
let $r=r(\bm{\alpha})$ be the number of zero differences, and
${m=m(\bm{\alpha})=\max\limits_{1\leq k\leq d}\alpha_k=\alpha^*_d}.$

We have
\begin{center}
	\begin{tabular}{|l|c|c|c|}\hline
		& ${m\in[0,2)}$&${m=2}$& $m>2$\\ \hline 
		$I_{\square}({\bf t})$ & ${\mathcal{O}({\bf t}^{-\bm\alpha})}$& ${\mathcal{O}({\bf t}^{-\bm\alpha}\ln^{r+1}({\bf t}^{\bm\alpha}))}$& ${\mathcal{O}({\bf t}^{-2{\bm\alpha}/m}\ln^{r}({\bf t}^{\bm\alpha}))}$\\ \hline
	\end{tabular}
\end{center}
\begin{center}
	\begin{tabular}{|l|c|c|r|}\hline
		& ${\theta\in(-(d+1),0]}$&${\theta=-(d+1)}$& $\theta<-(d+1)$\\ \hline
		$I_{\bigcirc}({\bf t})$ & ${\mathcal{O}({\bf t}^{-\bm\alpha})}$& ${\mathcal{O}({\bf t}^{-\bm\alpha}\ln({\bf t}^{\bm\alpha}))}$& ${\mathcal{O}({\bf t}^{{\bm\alpha}(d+1)/\theta})}$\\ \hline
	\end{tabular}
\end{center}

Already in the case of dimension ${d=2}$ interesting pictures arise: the parameter domain for
$I_{\square}({\bf t})$ is divided into 5 sets, and for $I_{\bigcirc}({\bf t})$ into three (see Figure~1).
\begin{figure}
	\begin{center}
		\includegraphics[scale=0.6]{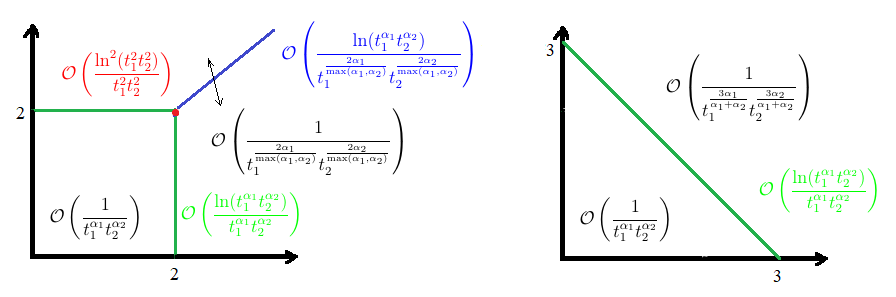} 
		{\caption{\small Partitioning the parameter domain due to estimates for $I_\square$ (left) and for $I_\bigcirc$ (right)}}
	\end{center}
\end{figure}
\medskip
By comparing the estimates, we can identify areas of parameters where the rate is better for $I_{\square},$ and where for $I_{\bigcirc}$ (see Figure~2).

\begin{figure}
	\begin{center}
		\includegraphics[scale=0.6]{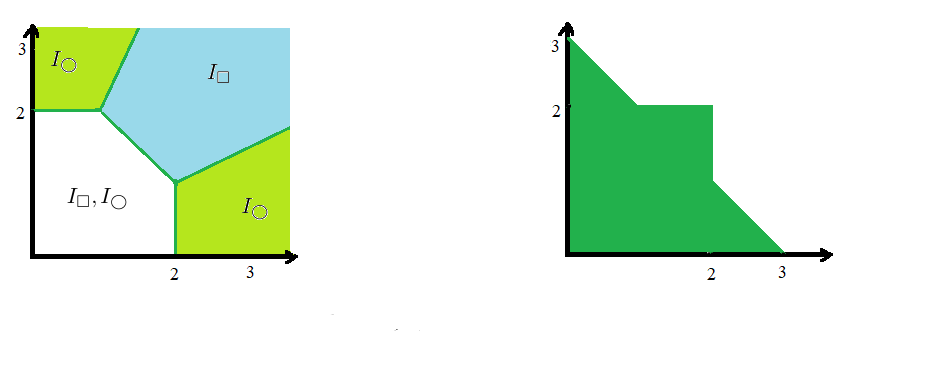} 
		{\caption{\small Partitioning of the parameter domain according to the choice of the best estimate between $I_\square$ and $I_\bigcirc$ (left); the set containing in the set of power estimates (right)}}
	\end{center}
\end{figure}

For a fixed vector $h\in\mathcal{H}$ with a power spectral asymptotics for neighborhoods of zero of the form~\eqref{eq_PowerEstimate}, we call the set of parameters \textit{a set of power estimates} if there exists~$\mathcal{K}\subset\mathbb{R}^d$ such that $I_\mathcal{K}({\bf t})=\mathcal{O}({\bf t}^{-\bm\alpha}).$ It is interesting to find out how such a set and its boundary are structured. Combining the power estimates for classical average of unitary actions of the group~$\mathbb{R}^2$, we see that the set of power estimates contains the green region (see Figure 2).

\vspace{1mm}
{\bf Conflict of Interest} {\rm The authors declare no conflict of interest.}
\par\vspace{1mm}




 {Russia, Novosibirsk, Sobolev Institute of Mathematics\\
E-mail\,$:$ ipodvigin@math.nsc.ru}

\end{document}